\documentclass[11pt,a4paper]{amsart}

\numberwithin{equation}{section}
\usepackage{mathtools}
\usepackage{amsmath,amssymb}
\usepackage{amsthm}
\usepackage{enumerate}
\usepackage{theoremref}

\theoremstyle{plain}
\newtheorem{theorem}{Theorem}[section]
\newtheorem{lemma}[theorem]{Lemma}
\newtheorem{proposition}[theorem]{Proposition}
\newtheorem{corollary}[theorem]{Corollary}

\theoremstyle{definition}
\newtheorem{definition}[theorem]{Definition}
\newtheorem{example}[theorem]{Example}
\newtheorem{remark}[theorem]{Remark}

\DeclareMathOperator{\Id}{Id}
\newcommand{\R}{\mathbb{R}}
\newcommand{\C}{\mathbb{C}}
\newcommand{\Z}{\mathbb{Z}}
\newcommand{\hm}{g}
\newcommand{\nm}{g'}
\newcommand{\tm}{g''}
\newcommand{\del}{\partial}
\newcommand{\holga}{\Gamma^{\mathrm{hol}}}
\newcommand{\seqref}[1]{{\it (\ref{#1})}}
\newcommand{\remarksymbol}{\hfill $\diamond$}

\title[Submanifolds of Wick-related spaces]{On holomorphic Riemannian geometry and submanifolds of Wick-related spaces}

\author{Victor Pessers and Joeri Van der Veken}
\address{
 KU Leuven, Department of Mathematics, 
 Celestijnenlaan 200B - Box 2400,
 3001 Leuven, Belgium}
\email{victor.pessers@wis.kuleuven.be}
\email{joeri.vanderveken@wis.kuleuven.be}


\thanks{This research was partially supported by the Belgian Interuniversity Attraction Pole P07/18 (Dygest).}

\begin{document}

\begin{abstract}
In this article we show how holomorphic Riemannian geometry can be used to relate certain submanifolds in one pseudo-Riemannian space to submanifolds with corresponding geometric properties in other spaces. In order to do so, we shall first rephrase and extend some background theory on holomorphic Riemannian manifolds, which is essential for the later application of the presented method.
\end{abstract}

\maketitle

\section{Introduction}

In this article, we show how certain problems in (pseudo-)Riemannian submanifold theory, that are situated in different ambient spaces, can be related to each other by translating the problem to an encompassing holomorphic Riemannian space. Our approach seems new in the area of submanifold theory, although it incorporates several existing insights, such as the theory on analytic continuation, complex Riemannian geometry and real slices, as well as the method of Wick rotations, which is mainly used in physics.

The relation between pseudo-Riemannian geometry and complex analysis can be traced back to the very birth of pseudo-Riemannian geometry. In the early publications on Lorentzian geometry by Poincar\'e and Minkowski (cf. \cite{Po,Mi}), the fourth coordinate of space-time was represented as $it$ (or $ict$), so that space-time was essentially modelled as $\R^3\times i\R$, where the standard complex bilinear form played the role of metric. Likely due to a later reformulation by Minkowksi himself, this point of view soon fell in abeyance in favor of the nowadays more common presentation in terms of an indefinite real inner product. Admittedly, as long as ones attention is kept restricted to four-dimensional Minkowski space alone, the use of complex numbers to deal with the signature bears little advantage.

This relationship between space-time geometry and complex numbers received renewed attention, when it was shown by Wick how problems from the Lorentzian setting are turned into problems in a Euclidean setting, after a so-called Wick rotation is applied on the time coordinate (cf. \cite{Wi}). This method of Wick rotations lays at the basis of the later theory on Euclidean quantum gravity, developed by Hawking and Gibbons among others (cf. \cite{Ha}), but also in other areas of theoretical physics it remained a valuable tool ever since.

Despite that the concept of Wick rotations is known by physicists for quite some time already, there are still many research domains where such insights have not been fully exploited yet. With this article, we hope to demonstrate that also for the particular area of submanifold theory, one may benefit from taking a complex viewpoint to problems about submanifolds of pseudo-Riemannian spaces. It should be noted that the subject of complex Riemannian manifolds, first introduced in an also physically motivated article by LeBrun (cf. \cite{Le}), has appeared in several articles on submanifold theory (e.g. \cite{BoFrVo,Sl}). 

The structure of this article is as follows. In Sections \ref{SC Preliminaries}, \ref{SC Complexifications} and \ref{SC Hol Riem submanifold theory}, we rephrase and extend the theory of holomorphic Riemannian manifolds from respectively a {\em complex-linear} perspective, a complex analytic perspective and subsequently the viewpoint of submanifold theory. This theory is then used in the \ref{SC applications}th and final section of this article, where we demonstrate by three different examples our method to relate certain kinds of submanifolds in one pseudo-Riemannian space, to submanifolds with corresponding geometric properties in other so-called Wick-related spaces.

\section{Preliminaries on complex Riemannian geometry}\label{SC Preliminaries}

\subsection{Complex vector spaces with holomorphic inner product}

\paragraph{}
In the following, we will assume all vector spaces to be finite dimensional. We adopt the following two definitions on subspaces of a complex vector space (cf. \cite{Bo}).

\begin{definition}
Let $V$ be a complex vector space. We call a real linear subspace $W\subset V$ {\em totally real} if $W\cap i\,W = \{0\}$ and {\em generic} if $W+ i\,W = V$.
\end{definition}

Remark that if $W$ is both generic and totally real, then its real dimension equals the complex dimension of $V$. 

\begin{definition}
A {\em holomorphic inner product space} is a complex vector space $V$ equipped with a non-degenerate complex bilinear form $g$.
\end{definition}

Given a holomorphic inner product on $V$ one can always choose an orthonormal basis. This means that any $n$-dimensional holomorphic inner product space can be identified with $\C^n$, equipped with the standard holomorphic inner product
\begin{equation}\label{EQ standard inner product}
g_0( X,Y ) = \ X_1 Y_1+\ldots+X_n Y_n,
\end{equation}
where $X=(X_1,\ldots,X_n)$ and $Y=(Y_1,\ldots,Y_n)$. The name `holomorphic inner product' comes from the fact that, unlike the more used sesquilinear inner product, the above product is a holomorphic function from $\C^n\times \C^n$ to $\C$. 

We note that, given a complex-linear operator $A:V \to V$, there does not necessarily exist a basis of $V$ consisting of eigenvectors of $A$, not even when $A$ is symmetric with respect to a holomorphic inner product on $V$. However, if there exists a basis of eigenvectors of a symmetric operator $A$, there also exists an orthonormal basis of eigenvectors of $A$.

\begin{definition}\label{DF real slice hol in prod space}
Given a holomorphic inner product space $(V,g)$, we use the term {\em real slice} to denote a real linear subspace $W\subset V$, for which $g|_W$ is non-degenerate and real valued, i.e., $g(X,Y)\in \R$ for all $X,Y\in W$.
\end{definition}

\begin{example}\label{EX pseudo-Euclidean}
  Consider the standard holomorphic inner product space $(\C^n,g_0)$ and let $\{e_1,\ldots,e_n\}$ be the standard basis. The real linear subspace
  \begin{equation}\label{EQ pseudo-Euclidean} 
  \R^n_k := \mathrm{span}_{\R}\{ie_1,\ldots,ie_k,e_{k+1},\ldots,e_n\}  
  \end{equation}
  is a real slice. Indeed, the restriction of the inner product $g_0|_{\R^n_k}$ is real valued, in particular, it is the standard pseudo-Euclidean metric of signature $k$ on $\R^n$, as the notation $\R^n_k$ suggests.
  If $z_j = x_j + iy_j$ are the standard coordinates on $\C^n$, the space $\R^n_k$ is given by $x_1 = \ldots = x_k = y_{k+1} = \ldots = y_n = 0$. Hence, $(y_1,\ldots,y_k,x_{k+1},\ldots,x_n)$ are natural real coordinates on $\R^n_k$.  
\end{example}
  
  Since any $n$-dimensional holomorphic inner product space $(V,g)$ can be identified with $(\C^n,g_0)$ by choosing an orthonormal basis, we can always find a real slice of any signature $k$, with $0 \leq k \leq n$ in $V$, namely the one corresponding to $\R_k^n \subset \C^n$. One easily verifies that all real slice of $V$ come from such an identification, in particular, they are all related by the action of the complex orthogonal group $\mathrm{O}(n,\C)$ on $V$.

\begin{proposition}\label{ST real slices hol in prod space}
The real slices of a holomorphic inner product space are totally real subspaces.
\end{proposition}
\begin{proof}
Let $W$ be a real slice of a holomorphic inner product space $(V,g)$, and let $X\in W\cap i\,W$. We need to show that $X=0$. Since $X\in iW$, there exists a vector $X'\in W$ such that $X=i X'$. Then we have that for all $Y\in W$, both $\hm(X,Y)$ and $\hm(X',Y)=-i\hm(X,Y)$ are real. Hence $\hm(X,Y)=0$ for all $Y\in W$, so that the non-degeneracy of $\hm$ implies $X=0$.
\end{proof}

\subsection{Complex-linear differential geometry}

\paragraph{}

In the following, we are concerned with complex manifolds. 
Importantly, we will regard complex manifolds from the perspective of {\em complex-linear differential geometry}, which may be characterized by the fact that all linear algebra involved is assumed to have $\C$ rather than $\R$ as its underlying field. This means, for instance, that a vector basis is usually assumed to be linearly independent over $\C$, rather than over $\R$, and tensors over the holomorphic tangent bundle are assumed to have $\C$ as their underlying field.

Indeed, the tensors (of any mixed type) we will be considering, all turn out to be tensorial over $\C$ in each of their components. In particular, we have that those tensors are all {\em pure} in the sense of \cite{YaAk}. Note that such assumptions cannot be imposed in the context of Hermitian manifolds for instance, for which even the main metric tensor is not complex bilinear.

\paragraph{}

Apart from being $\C$-linear, the tensor fields we will be concerned with are usually also holomorphic. We recall the following definition of holomorphic vector fields in a general holomorphic vector bundle:

\begin{definition}
Let $\pi:V\to M$ be a holomorphic vector bundle. A section $X$ of $V$ is called a {\em holomorphic vector field}, if $X$ is holomorphic as a map between the complex manifolds $M$ and $V$. The space of holomorphic sections of $V$ is denoted by $\holga(V)$.
\end{definition}

It follows directly that a vector field is holomorphic if and only if it has holomorphic component functions with respect to any local complex coordinates. For the special case of $V = T_l^k M$, by which we mean the tensor bundle of type $(l,k)$ over the holomorphic tangent bundle $TM$, this means that a tensor field $F\in \Gamma(T_l^k M)$ is holomorphic if and only if the component functions $F_{i_1\ldots i_l}^{j_1\ldots j_k}$ are holomorphic relative to any local holomorphic coordinates $\{z_1,\ldots,z_n\}$ of $M$.

The space of holomorphic tensors fields is closed under the usual tensor operations such as addition, tensor multiplication and contraction.

\subsection{Complex and holomorphic Riemannian geometry}\label{SC complex Riemannian geometry}

\paragraph{}

In this section we will review some basic theory on complex Riemannian manifolds, which we will rephrase here from the viewpoint of the complex-linear setting.

We will define a {\em complex Riemannian manifold} as a complex manifold $M$ endowed with a {\em complex Riemannian metric}: a non-degenerate symmetric tensor $\hm\in\Gamma(T_2^0 M)$, which means that it is a complex bilinear form on every tangent space (cf. \cite{Le}).

\paragraph{}
Like in the case of ordinary (pseudo-)Riemannian manifolds, any complex Riemannian manifold can be equipped with a Levi-Civita connection $\nabla$, which is the unique affine connection that is torsion free and compatible with the complex Riemannian metric, i.e.,
\begin{align}
[X,Y] &= \nabla_X Y - \nabla_Y X,\\
X(\hm(Y,Z)) &= \hm(\nabla_X Y,Z)+\hm(Y,\nabla_X Z).
\end{align}
for all vector fields $X$,$Y$ and $Z$. In the following, we will focus our attention to complex Riemannian manifolds for which the complex metric is holomorphic, which are called {\em holomorphic Riemannian manifolds} (cf. \cite{DuZe}). 

\begin{remark}
As has been pointed out in \cite{BoFrVo}, there is a direct correspondence between $n$-dimensional holomorphic Riemannian manifolds and anti-K\"ahler manifolds (also called K\"ahler-Norden manifolds): by restricting the holomorphic metric $g$ to its real part $g'$, our manifold may be regarded as a real $2n$-dimensional pseudo-Riemannian manifold with canonical complex structure $J$ and a Levi-Civita connection $\nabla$ induced by $g'$, satisfying:
$$g'(JX,JY)=-g'(X,Y),\quad \mbox{and}\quad \nabla_X JY = J \nabla_X Y,$$
which are the characteristic properties of an anti-K\"ahler manifold. 
On the other hand, any anti-K\"ahlerian metric (also called Norden metric) can be obtained from a uniquely determined holomorphic metric in this manner. The metric $g$ may be written as a direct sum of its real and imaginary part as:
\begin{equation}
  \hm = \nm + i \tm
\end{equation}
where $\nm$ is the (primary) Norden metric and $\tm$ is the secondary Norden metric (also called the {\em twin metric}). These Norden metrics are related to each other by the relation:
\begin{equation}
\nm(X,Y) = \tm(X,JY),\quad \mbox{and}\quad \tm(X,Y) = -\nm(X,JY)
\end{equation}
for $X,Y\in T M$. \remarksymbol
\end{remark}

The easiest example of a holomorphic Riemannian manifold is the complex Euclidean $n$-space $\C^n$. The tangent space in a point being naturally identified with the vector space $\C^n$, the standard holomorphic metric on $\C^n$ is given by \eqref{EQ standard inner product}.

\section{Complexifications and real slices}\label{SC Complexifications}

\subsection{Complexifications of totally real submanifolds}

\begin{definition}
Let $Q$ be a real analytic manifold and $M$ a complex manifold. An immersion $f:Q\to M$ is called {\em totally real} at a point $p\in Q$, if $TQ|_p$ is a totally real subspace of $TM|_p$ (i.e. $TQ|_p \cap i\,TQ|_p = \{0\}$). It is called {\em generic} at a point $p\in Q$, if $TQ|_p$ is a generic subspace of $TM|_p$ (i.e. $TQ|_p + i\,TQ|_p = TM|_p$). A totally real immersion is an immersion that is totally real at all points, and likewise for a generic immersion.
\end{definition}

In the above definition and further on, we identify $TQ|_p$ and $df(TQ|_p)$. It should be noted that, although the above definition of totally real is common in the area of complex analysis (cf. \cite{Bo}), it differs from the meaning it is usually given in the area of submanifold theory, where one commonly sees the additional requirement that for each point $p\in Q$, the subspace $TQ|_p$ is perpendicular to the subspace $iTQ|_p$ (cf. \cite{ChOg}).

The following theorem will be useful for the proof of later results (cf. \cite{WhBr}).

\begin{theorem}\label{ST unextendible extension}
Let $f:Q\to M$ be a totally real analytic immersion of a real analytic manifold $Q$ into a complex manifold $M$. Then a complexification $\C Q$ of $Q$ exists, such that $Q$ is a generic totally real submanifold of $\C Q$, and a holomorphic extension $\C f:\C Q\to M$ of $f$ exists, such that $\C f|_Q = f$. Such a holomorphic extension is unique in the sense that any other holomorphic extension will locally have the same image in $M$.
\end{theorem}

The complex space $\C Q$ in the above theorem may be related to the notion of complexifications of an abstract manifold $Q$, as for instance in \cite{Ku}. However, it should be noted that in our situation the space $\C Q$ depends on the given immersion $f:Q\to M$.

\begin{example}
Consider the following immersions:
$$f_1:\R\to \C^2: t\mapsto (t,0),\quad f_2:\R\to \C^2: t\mapsto (\cosh t,\sinh t).$$ Although $f_1$ and $f_2$ have images both homeomorphic to $\R$, the following holomorphic extensions 
$$\C f_1:\C\to \C^2: z\mapsto (z,0),\quad \C f_2:\C/(2\pi i\Z)\to \C^2: z\mapsto (\cosh z,\sinh z)$$ 
have images that are not homeomorphic.
\end{example}

Theorem \ref{ST unextendible extension} implies the following result, which we will need in the rest of this article.

\begin{corollary}\thlabel{ST analyt cont vector field}
Let $\pi:V\to M$ be a holomorphic vector bundle, and let $X\in \Gamma(V|_Q)$ be a real-analytic section over a real-analytic totally real submanifold $Q\subset M$. Then there exist a complexification $\C Q\subset M$ of $Q$ and a unique holomorphic extension $\C X\in \holga(V|_{\C Q})$ of $X$, which is a holomorphic section over $\C Q$.
\end{corollary}
\begin{proof}

First we observe that $X:Q\to V$ is a totally real immersion, for suppose that there is a $p\in Q$ and $v,w\in TQ|_p$, such that $dX(v)=i\,dX(w)$. We have that $\pi \circ X = \Id$, hence $d\pi \circ dX = \Id$, the identity operator on $TQ$. Using furthermore that $d\pi$ is complex-linear, we see that
$$v = (d\pi\circ dX)(v) = d\pi(i\,dX(w)) = i\,(d\pi\circ dX)(w) = i\,w,$$ 
from it which follows that $v=w=0$, and thus $dX(v)=dX(w)=0$, proving that $X$ is a totally real immersion into $V$. Hence, by Theorem \ref{ST unextendible extension}, there exists a complexification $\C Q$ of $Q$ and a unique holomorphic extension $\C X: \C Q\to V$ of $X$. Now it remains to prove that $\C X$ is a section over $\C Q$, i.e. $\pi\circ \C X = \Id|_{\C Q}$. This follows from the fact that $\pi\circ \C X$ equals the identity on $Q$, so that both $\pi\circ \C X$ and $\Id|_{\C Q}$ are a holomorphic extension of $\Id|_Q$. Hence, by the unicity part of Theorem \ref{ST unextendible extension}, it follows that $\pi\circ\C X = \Id|_{\C Q}$.
\end{proof}

\subsection{Real slices}

\begin{definition}\label{DF real slice}
Given a complex manifold $M$ with complex Riemannian metric $\hm$, we call a submanifold $N\subset M$ a {\em real slice} of $(M,\hm)$, if at any point $p\in N$ we have that $TN|_p$ is a real slice of $(TM|_p,g)$ in the sense of Definition \ref{DF real slice hol in prod space}.
\end{definition}

The above definition says that the metric $\hm$ restricted to the real tangent bundle of $N$, is real valued. Hence, the restriction $\hm|_{TN}$ turns $N$ into a pseudo-Riemannian manifold.
This notion of real slices is in accordance with the concept of real slices as introduced in \cite{Le}. It is easily seen that a submanifold $N$ is a real slice if and only if the secondary Norden metric $\tm|_M$ vanishes entirely on $M$, implying that the induced complex metric $g|_{TN}$ coincides with the induced primary Norden metric $\nm|_{TN}$. 

The following two propositions are immediate consequences of Definitions \ref{DF real slice hol in prod space} and \ref{DF real slice} and Proposition \ref{ST real slices hol in prod space}.

\begin{proposition}\label{ST real slice}
A real slice of a complex Riemannian manifold is a totally real submanifold.
\end{proposition}

\begin{proposition}\label{ST submanifold of real slice}
Let $M$ be a complex Riemannian manifold and $P$ a real slice of $M$. Then any real submanifold $Q\subset P$ is also a real slice of $M$. 
\end{proposition}

This last proposition brings along the following useful corollary.

\begin{corollary}\label{ST intersection is real slice}
Let $M$ be a complex Riemannian manifold, $N$ a complex submanifold of $M$ and $P$ a real slice of $M$. Then the intersection $N\cap P$ is a real slice of $N$, provided it is non-empty.
\end{corollary}

It should be noted that for a given complex manifold $M$, only very specific pseudo-Riemannian manifolds may appear as generic real slices of it. For if $Q$ is a generic real slice of $M$, then a complexification $\C Q$ is an open subset of $M$. This fully determines the complex Riemannian space $M$ in a neighborhood around $Q$. We will make the following definition (cf. \cite{Br}).

\begin{definition}\label{DF Wick-related}
Two pseudo-Riemannian manifolds $P$ and $Q$ are said to be {\em Wick-related} if there exists a holomorphic Riemannian manifold $(M,\hm)$, such that $P$ and $Q$ are embedded as real slices of $M$.
\end{definition}

As will be discussed later in more detail, when two manifolds are Wick-related, their geometries are intimately related through their common complexified ambient space. In this regard, as many geometric properties may be expressed as the vanishing of a certain tensor, the following two lemmas will turn out to be useful for the applications in Section \ref{SC applications}. First we have the following very trivial observation.

\begin{lemma} \label{ST from complex-analogous to pseudo-Riem}
If $N$ is a complex $n$-dimensional submanifold of a holomorphic Riemannian manifold $M$ for which a certain holomorphic tensor field $F$ vanishes, then for any real slice $P$ of $M$ for which the intersection $Q=N\cap P$ is a real $n$-dimensional submanifold, we have that $F|_Q$ vanishes on $Q$ as well.
\end{lemma}

The second lemma follows from \thref{ST analyt cont vector field}, and is in some some sense a converse to Lemma \ref{ST from complex-analogous to pseudo-Riem}.

\begin{lemma} \label{ST from pseudo-Riem to complex-analogous}
Let $M$ be a holomorphic Riemannian manifold, and let $P$ be a generic real slice of $M$. Then given a real-analytic pseudo-Riemannian submanifold $Q\subset P$ on which a particular real-analytic tensor $F$ vanishes, then there exists a complexification $N=\C Q\subset M$ for which the holomorphic extension $\C F$ of $F$ vanishes on $N$.
\end{lemma}

\section{Holomorphic Riemannian submanifold theory}\label{SC Hol Riem submanifold theory}

We will start this section with a brief overview of {\em holomorphic Riemannian submanifold theory}, the holomorphic Riemannian counterpart of (pseudo-)Riemannian submanifold theory. As we will see shortly, in the setting of holomorphic Riemannian submanifold theory, all tensors and operators of interest turn out to be holomorphic in nature, meaning that they are respectively either holomorphic tensors, or operators that result in a holomorphic tensor if all of their arguments are holomorphic. This is an important property, as it allows for holomorphic extensions of any vector field constructed by them.

\paragraph{}

It is well-known that on a complex manifold, given holomorphic vector fields $X,Y\in \holga(M)$, we have that $[X,Y]$ is again a holomorphic vector field, and $[X,cY]=[cX,Y]=c[X,Y]$ for any $c\in\C$ (cf. \cite{KoNo}). 

On a holomorphic Riemannian manifold $(M,g)$, the metric $g$ is a holomorphic tensor by definition, and it is a complexification of the pseudo-Riemannian metric $g|_Q$ for any real slice $Q\subset M$. It follows from Koszul's formula for the holomorphic Riemannian Levi-Civita connection that, given $X,Y\in \holga(TM)$, $\nabla_X Y$ is a holomorphic vector field as well. Furthermore, for a holomorphic function $f$ and $c\in \C$, we have that $\nabla_{f X} Y = f \nabla_X Y$ and $\nabla_{X} c Y = c \nabla_X Y$ (cf. \cite{Sl}).

The holomorphic Riemannian curvature tensor is defined by
$$ R(X,Y)Z = \nabla_X\nabla_Y Z - \nabla_Y\nabla_X Z - \nabla_{[X,Y]}Z.$$
It follows from a simple calculation that this indeed defines a $\C$-linear tensor. This tensor is moreover holomorphic by the previous observations on the holomorphy of the Lie bracket and the Levi-Civita connection. 

We can then use $R$ to define a holomorphic Riemannian analogue of the sectional curvature as follows:
\begin{equation}\label{EQ sectional curvature}
K: \holga(T^2M^{\mathrm{n.d.}})\to \C: X \otimes Y \mapsto \frac{\hm(R(X,Y)Y,X)}{\hm(X,X)\hm(Y,Y)-\hm(X,Y)^2},
\end{equation}
where $T^2M^{\mathrm{n.d.}}$ denotes the subbundle of $(2,0)$-tensors $X\otimes Y$ in $T^2M$, for which $X$ and $Y$ span a non-degenerate complex plane, i.e., for which the denominator in \eqref{EQ sectional curvature} is non-zero. Then $K$ is a holomorphic map and this definition corresponds to the definition of complex sectional curvature in \cite{DoKoPe}. 

In similar ways, we can define the holomorphic Riemannian analogues of the Ricci-curvature, the scalar curvature, etc. As long as these holomorphic Riemannian analogues of common pseudo-Riemannian quantities are defined through operations that preserve holomorphy, the result will be holomorphic as well.

We will now continue with defining some extrinsic tensors and operators for submanifolds in the holomorphic Riemannian setting. Let $(M,\hm)$ be a holomorphic Riemannian manifold and $N$ a holomorphic Riemannian submanifold. We denote the Levi-Civita connection of the ambient space $M$ by $D$, and the induced Levi-Civita connection on the submanifold $N$ by $\nabla$.

\begin{proposition} \label{prop:topandbot}
The projections $\cdot^\top$ and $\cdot^\bot$ of $TM|_N =TN\oplus (TN)^\bot$ to $TN$ and $(TN)^\bot$ respectively, are both holomorphic tensor fields.
\end{proposition}
\begin{proof}
For any point $p\in N$, it is possible through a Gram-Schmidt orthonormalization process to construct a holomorphic orthonormal frame $\{X_1,\ldots,X_n\}$ on some neighborhood $U$ of $p$ in $N$. 
Then for any section $Y \in \holga(TM|_U)$, one has
\begin{equation*} 
Y^\top = \sum_{i=1}^n \hm(X_i,Y)X_i. 
\end{equation*}
From this expression, it is seen that the tangent projection $\cdot^\top$ is a holomorphic tensor, and consequently, that the orthogonal projection $\cdot^\bot=\Id-\cdot^\top$ is a holomorphic tensor as well.
\end{proof}

An important consequence of Proposition \ref{prop:topandbot} is the following.

\begin{proposition}
The holomorphic Riemannian analogues of the second fundamental form $h$ and the shape operator $A$ are holomorphic tensors.
\end{proposition}
\begin{proof}
Let $M$ and $N$ be as above, and let $X$ and $Y$ be vector fields tangent to $N$ and $\xi$ a vector field normal to $N$. In the holomorphic Riemannian setting, the formulas of Gauss and Weingarten read identical the ones in the (pseudo-)Riemannian setting, which is:
\begin{align}\label{EQ Gauss and Weingarten}
D_X Y &= \nabla_X Y + h(X,Y),\\
D_X \xi &= -A_\xi X + \nabla_X^\bot \xi.
\end{align}
Hence $h$ and $A$ can be expressed as: 
\begin{equation} \label{eq:handA}
h(X,Y) = (D_X Y)^\bot,\qquad A_{\xi} X = -(D_X \xi)^\top.
\end{equation}
From the previous proposition, it immediately follows that $h$ and $A$ are holomorphic, and tensoriality is verified by a straightforward calculation.
\end{proof}
Note that the orthogonal counterparts of the equations \eqref{eq:handA} are
\begin{equation*}
\nabla_X Y = (D_X Y)^\top,\qquad \nabla_X^\bot \xi = (D_X \xi)^\bot,
\end{equation*}
from which can be seen that the induced connection on $N$, which is of course the Levi-Civita connection of $(N,g)$, and the normal connection are $\C$-linear holomorphic operators. 

Based on these fundamental extrinsic operators, many other extrinsic quantities can be defined. An important example of this is the {\em holomorphic Riemannian mean curvature} 
$$ H: N \to TM: p \mapsto \frac{1}{n} \mathrm{Tr}(h|_p) = \frac 1n \sum_{i=1}^n h(X_i,X_i) \in (TN)^{\bot}|_p, $$ 
where $X_1,\ldots,X_n$ is a orthonormal basis of $TN|_p$. Again, it is seen immediately that this defines a holomorphic tensor field.

\begin{remark}\label{RM holomorphic extension}
All of the holomorphic Riemannian tensors and operators discussed in this section, coincide with their pseudo-Riemannian counterparts when we restrict the (ambient) space to a generic real slice. By this we mean that if we take as input arguments some holomorphic extensions of real tensor fields over $Q$, then the output will be a holomorphic extension of the pseudo-Riemannian output on these real tensor fields. 
Indeed, the expressions through which these operators are defined are exactly the same as the way their pseudo-Riemannian counterparts would be defined on a real slice $Q$, except that the inputs are required to be holomorphic tensor fields. 
So in this sense, all the holomorphic Riemannian tensors and operators discussed above, are holomorphic extensions of their corresponding pseudo-Riemannian operators.
\end{remark}

\begin{remark}
Many well-known equations from (pseudo-)Riemannian submanifold theory, 
go through unaltered in the setting of holomorphic Riemannian submanifold theory. In particular we have got the equations of Gauss and Weingarten, i.e.:
\begin{align}
\label{EQ Gauss equation}g(\tilde R(X,Y)Z,W) &= g(R(X,Y)Z,W)+g(h(Y,W),h(X,Z))\\
 					&\qquad -g(h(X,W),h(Y,Z)),\nonumber\\
\label{EQ Codazzi equation}g(\tilde R(X,Y)\xi,\eta) &= g(R(X,Y)\xi,\eta)-g([A_\xi,A_\eta]X,Y),
\end{align}
where $\tilde R$ and $R$ denote the Riemann tensor on respectively the ambient manifold and the submanifold. These equations are identical to the equations in the ordinary (pseudo-)Riemannian case, as is the algebraic derivation through which these equations may be obtained (cf. \cite{Do}).
\end{remark}

\subsection{The standard holomorphic Riemannian space forms}

As mentioned above, an important example of a holomorphic Riemannian manifold is the complex Euclidean $n$-space $\C^n$. The curvature of this manifold vanishes and its isometry group is $E(n,\C) = \C^n \rtimes \mathrm O(n,\C)$, where $\C^n$ acts by translations. Some other important transformations of $\C^n$ are given by scalar multiplications with complex numbers $\alpha \neq 0$. It is clear that $L_\alpha: \C^n\to \C^n: z \mapsto \alpha z$ is a conformal transformation with conformal factor $\alpha^2$, i.e.,
$$ \hm|_{\alpha z}(d L_{\alpha}(X),d L_{\alpha}(Y)) = \alpha^2 \hm|_{z}(X,Y) $$
for all $z \in \C^n$ and all $X,Y\in T\C^n|_z$. In the particular cases that $\alpha=\pm 1$ and $\alpha = \pm i$, this gives respectively an isometry and an anti-isometry on the entire space. All dilations together with the group $E(n,\C)$ of isometries, generate the group of all conformal orthogonal transformations of the space $\C^n$, i.e. affine transformations which preserve orthogonality. In the following we will call such transformations {\em similarities}, and submanifolds that can be related to each other by a similarity are called {\em similar}.

Clearly, the generic real slices of $\C^n$ as a holomorphic inner product space (cf. Example \ref{EX pseudo-Euclidean}) are also generic real slices of $\C^n$ as a complex Euclidean space. Moreover, any translation of such a real slice $\R_k^n\subset \C^n$ over a complex vector $z\in\C^n$ results in an affine subspace $z+\R_k^n$, which is a real slice with the same signature. Of course, this space only differs from the former one if $z \notin V$. We will refer to such real slices of $\C^n$ as the {\em generic affine real slices}.

Another important example of a holomorphic Riemannian manifold is the holomorphic Riemannian sphere $\C S^n(\alpha)$ of radius $\alpha$, where $\alpha$ can be any non-zero complex number:
$$ \C S^n(\alpha) = \{ (z_1,\ldots,z_{n+1}) \in \C^{n+1} \ | \ z_1^2+\ldots+z_{n+1}^2 = \alpha^2 \}.$$ 
The holomorphic metric is induced by the holomorphic metric on $\C^{n+1}$. In the following, we will simply write $\C S^n$ for $\C S^n(1)$. All the $\C S^n(\alpha)$ are conformal since the dilation $L_\alpha:\C^{n+1}\to \C^{n+1}$ maps $\C S^n$ to $\C S^n(\alpha)$. The manifold $\C S^n(\alpha)$ has constant sectional curvature $1/\alpha^2$. In particular, the manifold $\C S^n(\alpha)$ is equal to the manifold $\C S^n(-\alpha)$, and anti-isometric to the manifolds $\C S^n(\pm i \alpha)$. As is suggested by the notation, the manifold $\C S^n$ is indeed a complexification of the sphere $S^n$ in $\R^{n+1}$.

As follows from Corollary \ref{ST intersection is real slice}, the intersection of $\C S^n$ with a real slice of $\C^{n+1}$ will be a real slice of $\C S^n$. If we consider the generic real slice $\R_k^{n+1}$ (cf. Example \ref{EX pseudo-Euclidean}) of $\C^{n+1}$ through the origin of $\C^{n+1}$, we obtain as real slice of $\C S^n$ the indefinite sphere (also called generalized de Sitter space)
\begin{equation*}
S^n_k = \{ (y_1,\ldots,x_{n+1}) \in \R^{n+1}_k \ | \ -y_1^2-\ldots-y_k^2+x_{k+1}^2+\ldots+x_{n+1}^2 = 1 \}.
\end{equation*}
For $k=0$, this is simply the $n$-dimensional round unit sphere $S^n$. 

It is known that the indefinite sphere $S_k^n$ is anti-isometric to the indefinite hyperbolic $n$-space (also called generalized anti-de Sitter space) $H_{n-k}^n$, where $H_k^n$ is defined by:
\begin{equation*}\label{EQ indef hyperbolic}
H^n_{k} = \{ (y_1,\ldots,x_{n+1}) \in \R^{n+1}_{k+1} \ | \ -y_1^2-\ldots-y_{k+1}^2+x_{k+2}^2+\ldots+x_{n+1}^2 = -1 \}.
\end{equation*}
That $S_k^n$ and $H_{n-k}^n$ are anti-isometric, can now be seen from the fact that by applying the anti-isometry $L_i$ to $\C^{n+1}$ (which also induces an anti-isometry between submanifolds and their image under $L_i$), we have that the submanifold $S_k^n$ is turned into $L_i S_k^n = L_i (\C S^{n} \cap \R_k^{n+1}) = L_i \C S^{n} \cap L_i \R_k^{n+1} = \C S^n(i) \cap L_i \R_k^{n+1}$. Now it follows from Equation \ref{EQ pseudo-Euclidean} that
$$L_i \R_k^{n+1} = \mathrm{span}_{\R}\{e_1,\ldots,e_k,i e_{k+1},\ldots,e_{n+1}\} \cong \R_{n-k+1}^{n+1},$$ so that $\C S^n(i) \cap L_i \R_k^{n+1}$ corresponds (after a rearrangement of coordinates) to $H_{n-k}^n$. Note that for $k=n$, we obtain the ordinary hyperbolic $n$-space.

So we have seen that on the one hand all the indefinite spheres are Wick-related, and on the other hand that all the indefinite hyperbolic spaces are Wick-related. Also, the indefinite spheres can be related with the indefinite hyperbolic spaces trough a dilation $L_i$ on the ambient space $\C^{n+1}$, which defines an anti-isometry between $S_k^n$ as a real slice of $\C S^n$, and $H_{n-k}^n$ as a real slice of $\C S^n(i)$. So, although the $n$-sphere is not directly Wick-related to hyperbolic $n$-space, they are in a sense Wick-related up to an anti-isometry. This means that a geometric problem in the $n$-dimensional sphere can be translated to an analogous problem in $n$-dimensional hyperbolic space, provided that the properties of interest are preserved under anti-isometric mappings. An easy example of this will be given in Section \ref{SC surfaces in CS^n}.

\section{Applications}\label{SC applications}

In this section we will demonstrate by some elementary examples in submanifold theory, how the above theory can be applied to reveal geometric correspondences among Wick-related manifolds. The main idea is as follows. As we know from Proposition \ref{ST real slice}, real slices are always totally real submanifolds. Hence, by \thref{ST analyt cont vector field} we have that any real-analytic tensor field over a generic real slice can be extended to a holomorphic tensor field on an open set of the ambient space. If this ambient space also contains other generic real slices, then by the uniqueness of the analytic extension, we have established a one-to-one correspondence between such tensor fields on these pseudo-Riemannian manifolds. As many geometric properties may be defined by the vanishing of a certain tensor, Lemmas \ref{ST from complex-analogous to pseudo-Riem} and \ref{ST from pseudo-Riem to complex-analogous} allow us to translate knowledge about such geometric properties among Wick-related manifolds (cf. Definition \ref{DF Wick-related}). Moreover, by proving a statement for the ambient holomorphic Riemannian manifold, we obtain similar statements simultaneously for all the entailed real-analytic real slices of this manifold. 

In the next three examples, we consider minimal, totally geodesic and parallel submanifolds in respectively $\C^3$, $\C S^3$ and $\C^3$. As these geometric properties are all invariant under similarities of the ambient space, the submanifolds only need to be characterized up to similarity. 

\subsection{Minimal surfaces of revolution in $\C^3$}\label{EX catenoid}

It is well-known that in $\R^3$, any minimal surface of revolution is similar to either a plane, or to the catenoid $C\subset \R^3$, given by
$$ C = \{ (x_1,x_2,x_3)\in\R^3 \ | \ x_1^2+x_2^2=\cosh^2 x_3 \}. $$
From this equation, it is indeed easily seen that $C$ is invariant under an $SO(2,\R)$ action, where $SO(2,\R)$ is represented as the subgroup of $SO(3,\R)$ given by the matrices of the form
\begin{equation}\label{EQ matrix revol}
\begin{pmatrix}
\cos \theta & -\sin \theta & 0 \\ 
\sin \theta & \cos \theta & 0 \\ 
0 & 0 & 1
\end{pmatrix}
\end{equation}
where $\theta\in\R$. Now by Lemma \ref{ST from pseudo-Riem to complex-analogous}, we obtain a minimal holomorphic Riemannian surface $\C C\subset \C^3$, given by
\begin{equation}\label{EQ catenoid in C3 1}
\C C = \{ (z_1,z_2,z_3)\in\C^3 \ | \ z_1^2+z_2^2=\cosh^2 z_3 \}. 
\end{equation}
Here and in the following, \emph{minimal} means that the (holomorphic Riemannian) mean curvature vector field $H$ vanishes everywhere. The surface \eqref{EQ catenoid in C3 1} is also easily seen to be invariant under the subgroup $SO(2,\C)\subset SO(3,\C)$, formed by the matrices \eqref{EQ matrix revol}, but now for $\theta\in \C$. This group contains both the Euclidean rotation group $SO(2,\R)$ as well as the indefinite group $SO(1,1,\R)$, given by all matrices of the form
$$\pm\begin{pmatrix}
\cosh \theta & \sinh \theta \\ 
\sinh \theta & \cosh \theta
\end{pmatrix}$$
where $\theta\in \R$. Note that the following four complex surfaces are all similar to $\C C$ in $\C^3$:
\begin{align}
& \{ (z_1,z_2,z_3)\in\C^3 \ | \ z_1^2+z_2^2=-\sinh^2 z_3 \}, \label{EQ catenoid in C3 2}\\
& \{ (z_1,z_2,z_3)\in\C^3 \ | \ z_2^2+z_3^2=\cosh^2 z_1 \}, \label{EQ catenoid in C3 3}\\
& \{ (z_1,z_2,z_3)\in\C^3 \ | \ z_1^2+z_2^2=-\sin^2 z_3 \}, \label{EQ catenoid in C3 4}\\
& \{ (z_1,z_2,z_3)\in\C^3 \ | \ z_2^2+z_3^2=-\sin^2 z_1 \}. \label{EQ catenoid in C3 5}
\end{align}
The surface \eqref{EQ catenoid in C3 2} is obtained from \eqref{EQ catenoid in C3 1} by a translation over a distance $\pi i/2$ in the direction of the $z_3$ axis and the surface \eqref{EQ catenoid in C3 4} is then obtained from \eqref{EQ catenoid in C3 2} by performing a dilation $L_i$ (which is an anti-isometry) on the ambient space $\C^3$. The surfaces \eqref{EQ catenoid in C3 3} and \eqref{EQ catenoid in C3 5} are obtained from \eqref{EQ catenoid in C3 1} and \eqref{EQ catenoid in C3 4} respectively, by permuting the coordinates $z_1$ and $z_3$. Hence \eqref{EQ catenoid in C3 3} and \eqref{EQ catenoid in C3 5} are rotation symmetric with respect to the $z_1$-axis. All these surfaces are mutually similar in the holomorphic Riemannian space $\C^3$, so there is no need to distinguish between them in $\C^3$. However, as we will see next, after intersecting them with a certain real slice, we may obtain surfaces that are not similar with respect to the similarities of this real slice (i.e., the group of affine transformations that preserve orthogonality within the real slice), thus giving rise to essentially different pseudo-Riemannian surfaces.

Of the five similar surfaces above, only \eqref{EQ catenoid in C3 1} and \eqref{EQ catenoid in C3 3} have non-empty intersection with the real slice $\R^3: y_1=y_2=y_3=0$, and in both cases we recover the classical catenoid. 

By intersecting the five surfaces with the real slice $\R_1^3:x_1=y_2=y_3=0$ on the other hand, we obtain the following five minimal surfaces (for the naming we follow \cite{An} and \cite{Wo}):
\begin{enumerate}
\item the Lorentzian hyperbolic catenoid, given by $-y_1^2+x_2^2=\cosh^2 x_3$,\label{LI Cat lor hyp}
\item the Lorentzian hyperbolic catenoid of the second kind, given by \linebreak[4] $y_1^2-x_2^2=\sinh^2 x_3$,\label{LI Cat lor hyp 2}
\item the Lorentzian elliptic catenoid, given by $x_2^2+x_3^2=\cos^2 y_1$,\label{LI Cat lor ell}
\item the spacelike hyperbolic catenoid, given by $y_1^2-x_2^2=\sin^2 x_3$,\label{LI Cat spa hyp}
\item the spacelike elliptic catenoid, given by $x_2^2+x_3^2=\sinh^2 y_1$.\label{LI Cat spa ell}
\end{enumerate}

Surfaces \eqref{LI Cat lor ell} and \eqref{LI Cat spa ell} are invariant under the group of Euclidean rotations (and reflections) $O(2,\R)\subset O(2,\C)$ whose axis of symmetry is the time-like $y_1$-axis. Surfaces \eqref{LI Cat lor hyp},\eqref{LI Cat lor hyp 2} and \eqref{LI Cat spa hyp} are invariant under the hyperbolic rotations (and reflections) $O(1,1,\R)\subset O(2,\C)$. Note that using only the property that the classical catenoid is a minimal surface in $\R^3$, we have obtained all of the other surfaces above, and their minimality is guaranteed without the need for any calculations.

\begin{remark}
In this and the following examples, the submanifolds are all described by implicit equations. This is done for convenience, but it should be noted that the possibility to use parametrizations is certainly there. For example, the six rotationally invariant minimal surfaces above (the one in $\R^3$ and the other five in $\R_1^3$), may respectively be parametrized as follows:
\begin{enumerate}\setcounter{enumi}{-1}
\item $L(u,v) = (\cos u \cosh v, \sin u \cosh v, v)$,\label{LI par Cat}
\item $L(u,v) = (i \sinh u \cosh v, \cosh u \cosh v, v)$,\label{LI par Cat lor hyp}
\item $L(u,v) = (i \cosh u \sinh v, \sinh u \sinh v, v)$,\label{LI par Cat lor hyp 2}
\item $L(u,v) = (i\, v, \cos u \cos v, \sin u \cos v)$,\label{LI par Cat lor ell}
\item $L(u,v) = (i \cosh u \sin v, \sinh u \sin v, v)$,\label{LI par Cat spa hyp}
\item $L(u,v) = (i\, v, \cos u \sinh v, \sin u \sinh v)$.\label{LI par Cat spa ell}
\end{enumerate}
It is worth noting how for example case \eqref{LI par Cat} is turned into case \eqref{LI par Cat lor ell} after multiplying the coordinate $u$ by $i$, a so-called Wick rotation. Just with some additional translations on the coordinates and suitable coordinate permutation on the ambient space, many of the other cases are obtained likewise.\remarksymbol
\end{remark}

\subsection{Totally geodesic surfaces in $\C S^n$}\label{SC surfaces in CS^n}

In this example we will see how the method of Wick-relations can be applied to situations where we have an ambient space other than $\C^n$. Before we come to this, it is good to realize that for a generically chosen holomorphic Riemannian manifold, one cannot even expect to find real slices at all (cf. \cite{Br}). Nevertheless, examples of holomorphic Riemannian manifolds that do contain real slices are easily constructed, as follows from Corollary \ref{ST intersection is real slice}: if $M$ is a holomorphic Riemannian submanifold of $\C^n$, which has a non-empty intersection with a certain real slice $\R_k^n$ of $\C^n$, then this intersection is itself a real slice of $M$.  We have already seen an application of this before, in obtaining $S_k^n$ as a real slice of $\C S^n$, and it can be applied to many other situations as well. For example, if we consider the complex surface $\C C \subset \C^3$ from the previous subsection, then the catenoid in $\R^3$ and the five minimal surfaces in $\R_1^3$ listed above are real slices of $\C C$, and thus Wick-related.

Just as a simple example, let us focus on 
$$ \C S^3 = \{ (z_1,z_2,z_3,z_4) \in \C^4 \ | \ z_1^2+z_2^2+z_3^2+z_4^2=1 \} $$
and consider the totally geodesic submanifold $\C S^2$ given by $z_1^2+z_2^2+z_3^2=1$ and $z_4=0$. We know this is a totally geodesic embedding by Lemma \ref{ST from pseudo-Riem to complex-analogous}, for if we intersect $\C S^2$ with the real slice $S^3\subset \C S^3$, given by $y_1=y_2=y_3=y_4=0$ (and hence $x_1^2+x_2^2+x_3^2+x_4^2=1$), we obtain the totally geodesic surface $S^2\subset S^3$, given by $x_1^2+x_2^2+x_3^2=1$ and $x_4=0$. From this, we know for sure that by intersecting $\C S^2$ (or submanifolds similar to it, such as $z_1^2+z_2^2+z_4^2=1$ and $z_3=0$ etc.) with one of the other real slices of $\C S^3$, we obtain totally geodesic submanifolds in those respective spaces as well. In this way, we find the following totally geodesic surfaces:
\begin{enumerate}
\item $S^2$ and $S_1^2$ in $S_1^3$, or equivalently $H_2^2$ and $H_1^2$ in $H_2^3$.
\item $S_1^2$ and $S_2^2$ in $S_2^3$, or equivalently $H_1^2$ and $H^2$ in $H_1^3$.
\item $S_2^2$ in $S_3^3$, or equivalently $H^2$ in $H^3$.
\end{enumerate} 
Note that the equivalent statement about embeddings of indefinite hyperbolic spaces is obtained through applying the anti-isometry $L_i$ on the ambient space $\C S^3$, which links each totally geodesic embedding of an indefinite $2$-sphere in an indefinite $3$-sphere to a totally geodesic embedding of some indefinite hyperbolic plane in an indefinite hyperbolic $3$-space. In particular, the totally geodesic embedding of the hyperbolic plane $H^2 : x_1^2+x_2^2-y_4^2=-1$ inside the hyperbolic space $H^3 :x_1^2+x_2^2+x_3^2-y_4^2=-1$ is thus obtained. Importantly, this simple example demonstrates how Wick-relationships can be used not only to reveal correspondences between pseudo-Riemannian spaces (possibly of different signature), but even to reveal correspondences between different kinds of ordinary Riemannian spaces (in this case the $n$-dimensional sphere and $n$-dimensional hyperbolic space).

\subsection{Parallel surfaces in holomorphic Riemannian $\C^3$}\label{EX Par}

In the following, we will classify surfaces in $\C^3$ with parallel second fundamental form, i.e., surfaces for which the holomorphic tensor $\nabla h$ vanishes identically. These are called parallel surfaces for short. 

\begin{lemma}\label{ST forms complex matrix}
Let $A$ be a linear operator on a complex two-dimensional vector space $V$, which is symmetric with respect to a holomorphic inner product $g$ on $V$. Then there exists an orthonormal basis $\{e_1,e_2\}$ of $V$ such that, with respect to $\{e_1,e_2\}$, $A$ takes one of the following forms:
\begin{eqnarray}
A=\begin{pmatrix}
\alpha & 0 \\
0 & \beta \end{pmatrix}\quad\mbox{or}\label{EQ diagonalizable}\\
A = \begin{pmatrix}
\alpha+1 & i \\
i & \alpha-1 \end{pmatrix}.\label{EQ undiagonalizable}
\end{eqnarray}
\end{lemma}
\begin{proof}
Let $\{u_1,u_2\}$ be a any orthonormal basis of $V$. Since $A$ is symmetric, it can be written in the form 
$$A = 
\begin{pmatrix}
m+a & b \\
b & m-a 
\end{pmatrix}$$ 
with respect to $\{u_1,u_2\}$. Now consider a general orthonormal basis given by $e_1 = \cos t\,u_1 + \sin t\,u_2$ and $e_2 = -\sin t\,u_1 + \cos t\,u_2$, where $t$ is a complex number. We want to choose $t$ in such a way that $A$ takes one of the two forms above.

If $b=0$, we are done immediately, so let us assume $b\neq 0$.

Provided that $a \neq \pm b i$, we may choose $t$ a complex number such that $\cot 2t = a/b$, and one verifies that $A$ takes the form \eqref{EQ diagonalizable} with respect to $\{e_1,e_2\}$. (Note that $\cot t$ assumes all complex values except for $\pm i$.)

In case $a = b i$, we can choose $t$ such that $e^{-2ti}=a$, and $A$ will take the form \eqref{EQ undiagonalizable} with respect to $\{e_1,e_2\}$. In case $a = -b i$, we choose $t$ such that $e^{2ti}=a$ to get the same result.
\end{proof}

\begin{theorem}\label{ST classification hol riem}
A parallel surface $M$ in $\C^3$ is similar to an open part of one of the following four surfaces.
\begin{enumerate}
\item The complex plane $\C^2$, given by $z_3 = 0$.\label{case 1}
\item The complex sphere $\C S^2$, given by $z_1^2+z_2^2+z_3^2=1$.\label{case 2}
\item The cylinder $\C S^1\times \C$, given by $z_1^2+z_2^2=1$.\label{case 3}
\item The flat minimal surface $B$, given by $z_3 = (z_1+iz_2)^2$.\label{case 4}
\end{enumerate}
\end{theorem}
\begin{proof}
Assume $M$ is a complex submanifold with parallel holomorphic Riemannian second fundamental form $h$, and let $L:U\subset \C^2\to \C^3$ denote a local parametrization of $M\subset \C^3$. Let $e_3$ be a (holomorphic Riemannian) unit normal vector field. First suppose that with respect to a local orthonormal frame field $\{e_1,e_2\}$, the shape operator $A$ takes the form \eqref{EQ diagonalizable}. A direct computation shows that the surface is parallel if and only if $\alpha$ and $\beta$ are constant and $(\alpha-\beta)\omega_1^2=0$, where $\omega_1^2$ is a one-form defined by $\nabla_X e_1 = \omega_1^2(X) e_2$.

If $\alpha=\beta=0$, then the surface is totally geodesic, i.e. $h = 0$, which gives case \eqref{case 1}. 

If $\alpha=\beta\ne 0$, then the second fundamental form satisfies:
$$h(e_1,e_1)=\alpha e_3,\quad h(e_1,e_2)=0,\quad h(e_2,e_2)=\alpha e_3.$$
Thus it follows from the Gauss equation \eqref{EQ Gauss equation} that $K=\alpha^2$. By choosing (holomorphic) geodesic coordinates $(u,v)$ such that $g = du^2 + \cos^2(\alpha u) dv^2$, we have that
\begin{align*}
L_{uu} = \alpha e_3,\quad L_{uv} &= - \alpha \tan(\alpha u) L_v,\quad L_{vv} = \alpha \cos^2(\alpha u)e_3 + \frac \alpha 2 \sin(2 \alpha u) L_u,\\ 
&\quad D_{\del_u}e_3 = -\alpha L_u, \quad D_{\del_v}e_3 = -\alpha L_v.
\end{align*}
The solution of this system of equations is, up to translation, given by
\begin{align*}
L(u,v) &=  \frac{1}{\alpha}\left(\cos(\alpha u) \cos(\alpha v) w_1 +  \cos(\alpha u) \sin(\alpha v) w_2 + \sin(\alpha u) w_3\right)\\
e_3(u,v) &= -\cos(\alpha u) \cos(\alpha v) w_1 - \cos(\alpha u) \sin(\alpha v) w_2 - \sin(\alpha u) w_3
\end{align*}
where $\{w_1,w_2,w_3\}$ is an orthonormal basis of $\C^3$ (as holomorphic inner product space). After rescaling, we obtain case \eqref{case 2} of the theorem.

If $\alpha\neq \beta$, then $\omega_1^2 = 0.$ Hence $M$ is flat, so that by the equation of Gauss we have $\operatorname{det} A = \alpha \beta = 0$. Without loss of generality, we may assume $\beta = 0$. Now, choose coordinates $(u,v)$ on $M$ with $\del_u = e_1$ and $\del_v = e_2$. Then, the immersion $L$ satisfies
$$
L_{uu} = \alpha e_3,\quad L_{uv} = 0,\quad L_{vv} = 0,\quad D_{\del_u}e_3 = -\alpha L_u, \quad D_{\del_v}e_3 = 0.
$$
The solution of this system of equations is, up to translation, given by
\begin{align*}
L(u,v) &= \frac{1}{\alpha}\left(\cos(\alpha u) w_1 + \sin(\alpha u) w_2\right) + v\,w_3,\\
e_3(u,v) &= -\cos(\alpha u) w_1 - \sin(\alpha u) w_2,
\end{align*}
where $\{w_1,w_2,w_3\}$ is an orthonormal basis of $\C^3$. After applying a suitable similarity in $\C^3$, we obtain case \eqref{case 3} of the theorem.

Now suppose that with respect to an orthonormal frame field $\{e_1,e_2\}$, $A$ takes the form \eqref{EQ undiagonalizable}. Since the surface is parallel, we obtain that $\alpha$ is constant and $\omega_1^2 = 0$. Hence the surface is flat. But from the equation of Gauss, we obtain that the Gaussian curvature is given by $\operatorname{det} A = \alpha^2$, so that $\alpha=0$. If we take coordinates $(u,v)$ on $M$ with $\del_u = e_1$ and $\del_v = e_2$, then the formulas of Gauss and Weingarten yield the following system of equations:
\begin{align*}
\quad L_{uu} &= e_3,\quad L_{uv} = i e_3,\quad L_{vv} = - e_3,\\
\quad D_{\del_u}e_3 &= -L_u-i L_v, \quad D_{\del_v}e_3 = -i L_u+L_v.
\end{align*}
The solution of this system of equations is, up to translation, given by
\begin{align*}
L(u,v) &= \left(u-\frac 1 6 (u+i v)^{3}\right)w_1 + \left(v-\frac i 6 (u+i v)^{3}\right)w_2 + \frac 1 2 (u+i v)^{2} w_3,\\
e_3(u,v) &= -(u+i v) w_1 - (i u-v) w_2 + w_3,
\end{align*}
where $\{w_1,w_2,w_3\}$ is an orthonormal basis of $\C^3$. After applying a suitable similarity, we obtain case \eqref{case 4} of the theorem.
\end{proof}

As follows from Lemma \ref{ST from complex-analogous to pseudo-Riem} and Remark \ref{RM holomorphic extension}, any intersection of a parallel complex surface in $\C^3$ with a generic real slice results in a parallel surface in that real slice. On the other hand, from Lemma \ref{ST from pseudo-Riem to complex-analogous} it follows that all real-analytic parallel surfaces in any of the pseudo-Euclidean spaces $\R_k^n$ are obtained in this way. In this regard, it is interesting to compare our classification result above with the classification theorems of parallel surfaces in $\R^3$ and $\R_1^3$ from \cite{ChVa}. It is worth noticing, that for the proof of Theorem \ref{ST classification hol riem}, only minor adoptions to the original proof for $\R_1^3$ were necessary.

\begin{theorem}\label{ST par in R3}
A non-degenerate parallel surface in $\R^3$ is similar to an open part of one of the following three surfaces.
\begin{enumerate}[(i)]
\item The plane $\R^2$, given by $x_3 = 0$.
\item The sphere $S^2$, given by $x_1^2+x_2^2+x_3^2=1$.
\item The flat cylinder $S^1\times \R$, given by $x_1^2+x_2^2=1$.
\end{enumerate}
\end{theorem}

These three cases in Theorem \ref{ST par in R3} above arise from intersecting the complex surfaces \eqref{case 1}, \eqref{case 2} and \eqref{case 3} from Theorem \ref{ST classification hol riem} with the Euclidean real slice $\R^3\subset\C^3$, given by $y_1=y_2=y_3=0$. The surface \eqref{case 4}, or any surface similar to it in $\C^3$, does not give a real 2-dimensional intersection with $\R^3$.

\begin{theorem}
A non-degenerate parallel surface in $\R_1^3$ (with metric $ds^2=-dy_1^2+dx_2^2+dx_3^2$) is similar to an open part of one of the following eight surfaces.
\begin{enumerate}[(i)]
\item The Euclidean plane $\R^2$, given by $y_1 = 0$.\label{LI Par R13 case 1}
\item The Lorentzian plane $\R_1^1$, given by $x_2 = 0$.\label{LI Par R13 case 2}
\item The hyperbolic plane $H^2$, given by $y_1^2-x_2^2-x_3^2=1$.\label{LI Par R13 case 3}
\item The indefinite sphere $S_1^2$, given by $-y_1^2+x_2^2+x_3^2=1$.\label{LI Par R13 case 4}
\item The flat cylinder $H^1\times \R^1$, given by $y_1^2-x_2^2=1$.\label{LI Par R13 case 5}
\item The flat cylinder $S^1\times \R_1^1$, given by $x_2^2+x_3^2=1$.\label{LI Par R13 case 6}
\item The flat cylinder $S_1^1\times \R^1$, given by $-y_1^2+x_2^2=1$.\label{LI Par R13 case 7}
\item The flat minimal Lorentzian surface $M_1^2$, given by $x_3 = (y_1-x_2)^2$.\label{LI Par R13 case 8}
\end{enumerate}
\end{theorem}

As these surfaces are all real-analytic (in fact even quadratic), they must all be obtainable from Theorem \ref{ST classification hol riem}, by taking appropriate intersections with $\R_1^3$. Cases \seqref{LI Par R13 case 1} and \seqref{LI Par R13 case 2} are obtained by intersecting the planes $z_1=0$ and $z_2=0$ (both similar to \eqref{case 1} from Theorem \ref{ST classification hol riem}) with $\R_1^3$. The cases \seqref{LI Par R13 case 3} and \seqref{LI Par R13 case 4} are obtained by intersecting with the surfaces $-z_1^2-z_2^2-z_3^2=1$ and $z_1^2+z_2^2+z_3^2=1$ (both similar to \eqref{case 2}). The cases \seqref{LI Par R13 case 5}, \seqref{LI Par R13 case 6} and \seqref{LI Par R13 case 7} are obtained by intersecting with the surfaces $-z_1^2-z_2^2=1$, $z_2^2+z_3^2=1$ and $z_1^2+z_2^2=1$ (all similar to \eqref{case 3}). Finally, \seqref{LI Par R13 case 8} is obtained by intersecting with $z_3 = (z_2+i z_1)^2$ (similar to \eqref{case 4}).

\begin{remark}
As mentioned before, the examples above are mainly chosen with the purpose of exemplifying our methods. However, there are no real obstructions to apply the same principle to deal with other problems in submanifold theory. First of all, no restrictions are being put on the dimensions of the manifolds involved. Furthermore, when it comes to the geometric property in consideration, there is only the restriction that it is derived in a holomorphic manner from the main holomorphic operators in submanifold theory. Hence, besides the properties we have seen in the examples above, other possible properties include constant mean curvature, constant sectional curvature, Einstein, quasi-minimal (marginally trapped), (semi- or pseudo-)symmetry, (semi- or pseudo-)parallelity, and many more.
\end{remark}

\bibliography{Wick_related_spaces}
\bibliographystyle{abbrv}
\end{document}